\documentclass[a4paper,leqno,11pt]{amsart}

\usepackage{epsf,graphicx,epsfig}
\usepackage{amscd}
\usepackage{amsmath,latexsym,amssymb,amsthm}
\usepackage[nospace,noadjust]{cite}
\usepackage{textcomp}
\usepackage{setspace,cite}
\usepackage{lscape,fancyhdr,fancybox}
\usepackage[all,cmtip]{xy}
\usepackage{tikz}
\usetikzlibrary{shapes,arrows,decorations.markings}
\usepackage{graphicx}

\usepackage{color}
\usepackage{url}
\usepackage{enumerate}
\usepackage[mathscr]{euscript}

  \theoremstyle{plain}

\swapnumbers
    \newtheorem{thm}{Theorem}[section]
    \newtheorem{prop}[thm]{Proposition}
   
    \newtheorem{corollary}[thm]{Corollary}
    
    \newtheorem{subsec}[thm]{}
\theoremstyle{definition}
    \newtheorem{defn}[thm]{Definition}
    \newtheorem{exam}[thm]{Example}

\theoremstyle{remark}
     \newtheorem{remark}[thm]{Remark}

\title{}
\author{}
\date{}
\usepackage{amssymb}

\usepackage{hyperref}
\hypersetup{
	colorlinks,
	citecolor=blue,
	filecolor=black,
	linkcolor=blue,
	urlcolor=black
}

\begin{document}
\title{Multiplicative Nambu structures on Lie groupoids}



\author{Apurba Das}
\email{apurbadas348@gmail.com}
\address{Stat-Math Unit,
Indian Statistical Institute, Kolkata 700108,
West Bengal, India.}




\thispagestyle{empty}

\begin{abstract}
We study some properties of coisotropic submanifolds of a manifold with respect to a given multivector field. Using this notion, we  generalize the results of
Weinstein \cite{wein} from Poisson bivector field to Nambu-Poisson tensor or more generally to any multivector field. We also introduce the notion of Nambu-Lie groupoid generalizing the concepts of both Poisson-Lie groupoid and Nambu-Lie group. We show that the  infinitesimal version 
of Nambu-Lie groupoid is the notion of weak Lie-Filippov bialgebroid as introduced in \cite{bas-bas-das-muk}. Next we introduce coisotropic subgroupoids of a
Nambu-Lie groupoid and these subgroupoids corresponds to, so called coisotropic subalgebroids of the corresponding weak Lie-Filippov bialgebroid.

\end{abstract}
\maketitle


\vspace{0.5cm}
\section{Introduction}

In \cite{wein}, Weinstein introduced the notion of coisotropic submanifold of a Poisson manifold generalizing the notion of Lagrangian submanifold
of symplectic manifold. A submanifold $C$ of a Poisson manifold $(P, \pi)$ is called coisotropic, if $\pi^{\sharp}(TC)^0 \subset TC,$ or, equivalently
$\pi (\alpha, \beta) = 0$, for all $ \alpha, \beta \in (TC)^0$, where $(TC)^0$ is the conormal bundle of $C$. Moreover, Weinstein proved the following results.
\begin{enumerate}
\item A map $\phi : P_1 \rightarrow P_2$ between Poisson manifolds is a Poisson map if and only if its graph is a coisotropic submanifold of
$P_1 \times P_2^{-},$ where $P_2^{-}$ stands for the manifold $P_2$ with opposite Poisson structure.
\item If $\phi : P \rightarrow Q$ is a surjective submersion from a Poisson manifold $P$ to some manifold $Q$, then $Q$
has a Poisson structure for which $\phi$ is a Poisson map if and only if
$$\{(x,y)| \phi(x) = \phi(y)\} \subset P \times P$$ is a coisotropic submanifold of $P \times P^{-}.$
\end{enumerate}

To define the coisotropic submanifold of a Poisson manifold, one does not require  the Poisson tensor to be closed, that is, $[\pi, \pi] = 0,$
where $[~,~]$ denotes the Schouten bracket on multivector fields. Therefore, the notion of  coisotropic submanifolds make sense for any bivector field, or more generally, for any multivector field.
Explicitly, if $M$ is a smooth manifold and $\Pi \in \mathcal{X}^n(M) = \Gamma \bigwedge^n TM$ be an $n$-vector field on $M$, then a submanifold $C \hookrightarrow M$ is called coisotropic with respect to $\Pi$
if 
$$\Pi^{\sharp} ({\bigwedge}^{n-1}(TC)^0) \subset TC ~~\Leftrightarrow ~~\Pi (\alpha_1, \ldots, \alpha_n) = 0,~~ \mbox{for all}~~ \alpha_1, \ldots, \alpha_n \in (TC)^0,$$
where $\Pi^\sharp : \bigwedge^{n-1}T^*M \rightarrow TM$ is the bundle map induced by $\Pi$.

Nambu-Poisson manifolds are generalization of Poisson manifolds. Recall that a Nambu-Poisson manifold of order $n$ is a manifold $M$ equipped with an $n$-vector field $\Pi$ such that the induced bracket on functions satisfies the Fundamental identity (Definition \ref{nambu-poisson}). The $n$-vector field $\Pi$ of a Nambu-Poisson manifold is referred to as the associated Nambu tensor. Coisotropic submanifolds of a Nambu-Poisson manifold $M$ are those submanifolds which are coisotropic with respect to the  Nambu tensor $\Pi$. 
 
In the present paper, we study some basic properties of coisotropic submanifolds of a manifold with respect to a given multivector field
and generalize the results of Weinstein to the case of multivector field. More precisely, we prove the following results (Propositions \ref{nambu map-coiso} and \ref{coinduced-coiso}).
\begin{enumerate}
\item Let $(M, \Pi_M)$ and $(N, \Pi_N)$ be two manifolds with $n$-vector fields and $\phi: M \rightarrow N$ be a smooth map. Then $\phi_* \Pi_M = \Pi_N$ if and only if its graph
$$\text{Gr}(\phi) := \{(m, \phi(m))| m \in M\}$$
is a coisotropic submanifold of $M \times N$ with resoect to  $\Pi_M \oplus (-1)^{n-1} \Pi_N$.
\item Let $(M, \Pi_M)$ be a  manifold with an $n$-vector field and $\phi: M \rightarrow N$ be a surjective submersion. Then $N$ has an (unique) $n$-vector field $\Pi_N$ such that $\phi_* \Pi_M = \Pi_N$ if and only if 
$$R(\phi) := \{(x,y) \in M \times M|\phi(x)=\phi(y)\}$$
is a coisotropic submanifold of $M \times M$ with respect to $\Pi_M \oplus (-1)^{n-1}{\Pi}_M$.
\end{enumerate}

Poisson Lie group is a Lie group equipped with a Poisson structure such that the group multiplication map is a Poisson map. Equivalently, a Lie group equipped with a Poisson structure is a Poisson Lie group if the Poisson bivector field is multiplicative \cite{lu-wein}. These definitions have no natural extension when one wants to define Poisson groupoid. Nevertheless, the notion of coisotropic submanifolds of Poisson manifolds was used by Weinstein \cite{wein} to introduce the notion of Poisson groupoid. Recall that a Poisson groupoid is a Lie groupoid $G \rightrightarrows M$ with a Poisson structure on $G$ such that the graph of the groupoid (partial) multiplication map is a coisotropic submanifold of $G \times G \times G^{-}.$ 

In \cite{xu}, P. Xu gave an equivalent formulation of Poisson groupoid which generalizes the multiplicativity condition for Poisson Lie group. More generally, In \cite{pont-geng-xu}, the authors introduced the notion of multiplicative multivector fields on a Lie groupoid. Given a Lie groupoid $G \rightrightarrows M$, an $n$-vector field $\Pi \in \mathcal{X}^{n}(G)$ is called multiplicative, if the graph of the groupoid multiplication
is a coisotropic submanifold of $G \times G \times G$ with respect to $\Pi \oplus \Pi \oplus (-1)^{n-1} \Pi$. In this terminology, a Poisson groupoid is a Lie groupoid equipped with a
multiplicative Poisson tensor. 

In the present paper, we extend this approach to the case of Lie groupoid with a Nambu structure. We introduce the notion of a Nambu-Lie groupoid as a 
Lie groupoid with a Nambu structure $\Pi$ such that the Nambu tensor $\Pi$ is multiplicative (Definition \ref{nambu-lie groupoid}).
When $G$ is a Lie group, this definition coincides with the definition of Nambu-Lie group given by Vaisman \cite{vais}. Using results proved in \cite{pont-geng-xu} for mutiplicative multivector fields on Lie groupoid, we deduce the following facts which are parallel to the case of Poisson groupoid. 
Suppose $(G \rightrightarrows M, \Pi)$ is a Nambu-Lie groupoid, then 
\begin{enumerate}
\item $M \hookrightarrow G$ is a coisotropic submanifold of $G$;
\item the groupoid inversion map $i: G \rightarrow G$ is an anti Nambu-Poisson map; 
\item there is a unique Nambu-Poisson structure $\Pi_M$ on $M$ for which the source map is a Nambu-Poisson map (Proposition \ref{inverse-basenambu}).
\end{enumerate}

It is well known that for a Nambu-Poisson manifold $M$ of order $n$, the space of $1$-forms admits a
skew-symmetric $n$-bracket which satisfies the Fundamental identity modulo some restriction (\cite{vais , gra-mar , bas-bas-das-muk}). Moreover, the bracket on forms and the de-Rham differential of the manifold satisfy a compatibility condition similar to that of a Lie bialgebroid. This motivates the authors \cite{bas-bas-das-muk} to introduce a notion of weak Lie-Filippov bialgebroid of order $n$. If $M$ is a Nambu-Poisson manifold of order $n$, then $(TM, T^*M)$ provides such an example. Roughly speaking, a weak Lie-Filippov bialgebroid of order $n$ ($n > 2$) over $M$ is a Lie algebroid $A \rightarrow M$ together with a skew-symmetric $n$-ary bracket
on the space of sections of the dual bundle $A^* \rightarrow M$ and a bundle map $\rho : \bigwedge^{n-1}A^* \rightarrow TM$ satisfying some conditions (cf. Definition \ref{lie-fill-defn}). Moreover it is proved in \cite{bas-bas-das-muk} that, if $(A, A^*)$ is a weak Lie-Filippov bialgebroid of order $n$ over $M$, then
there is an induced Nambu-Poisson structure of order $n$ on the base manifold $M$.

In the present paper, we prove that weak Lie-Filippov bialgebroids are infinitesimal form of Nambu-Lie groupoids. 
Explicitly, if $C$ is a coisotropic submanifold of a Nambu-Poisson manifold $(M, \Pi)$, then we show that the $n$-ary bracket on the space of $1$-forms
on $M$ restricts to the sections of the conormal bundle $(TC)^0 \rightarrow C$ and the induced bundle map $\Pi^{\sharp} : \bigwedge^{n-1}T^*M \rightarrow TM$ 
maps $\bigwedge^{n-1}(TC)^0$ to $TC$ (Proposition \ref{coiso-n-bracket}). Therefore, if $G \rightrightarrows M$ is a Nambu-Lie groupoid of order $n$ whose Lie algebroid is 
$AG \rightarrow M,$ then as $M$ is a coisotropic submanifold of $G$, the space of sections of the dual bundle $A^*G \cong (TM)^0 \rightarrow M$ is equipped with a skew-symmetric $n$-ary bracket. 
Moreover, there is a bundle map $\bigwedge^{n-1} A^*G \rightarrow TM$ so that the pair $(AG, A^\ast G),$ with the above data, satisfies the conditions of a weak Lie-Filippov bialgebroid. Thus we prove the following (cf. Theorem \ref{nambu-grpd-bialgbd}). 
\begin{thm}
 Let $(G \rightrightarrows M , \Pi) $ be a Nambu-Lie groupoid of order $n$ with Lie algebroid $AG \rightarrow M$. Then $(AG, A^*G)$ forms a weak Lie-Filippov bialgebroid of order $n$ over $M$.
\end{thm}

Finally, we compare the Nambu-Poisson structures on the base manifold $M$ induced from the Nambu-Lie groupoid $G \rightrightarrows M$ and the weak Lie-Filippov bialgebroid $(AG, A^*G)$ (cf. Proposition \ref{compare-two-NP-structure}).

Next, we introduce the notion of a coisotropic subgroupoid $H \rightrightarrows N$ of a Nambu-Lie groupoid $(G \rightrightarrows M, \Pi).$
To study the infinitesimal form of a coisotropic subgroupoid, we introduce the notion of coisotropic subalgebroid of a weak Lie-Filippov bialgebroid.
Then we show that the Lie algebroid of a coisotropic subgroupoid $H \rightrightarrows N$ is a coisotropic subalgebroid of the corresponding weak Lie-Filippov bialgebroid $(AG, A^*G)$ (cf. Proposition \ref{coiso-subgrpd-coiso-subalgbd}). 

{\bf Organization.} The paper is organized as follows. In section 2, we recall basic definitions and conventions. In section 3, we study some properties of coisotropic
submanifolds of a manifold with respect to a given multivector field and in section 4 we introduce Nambu-Lie groupoids and study some of its basic properties.
In section 5, we show that the infinitesimal object corresponding to Nambu-Lie groupoid is weak Lie-Filippov bialgebroid and in section 6 we introduce coisotropic
subgroupoids of a Nambu-Lie groupoid and study their infinitesimal.\\
{\bf Acknowledgements.} The author would like to thank his Ph.D supervisor Professor Goutam Mukherjee for his guidance and carefully reading the manuscript.
\section{Preliminaries}
In this section, we recall some basic preliminaries from \cite{duf-zung , mac-book , xu} and fix the notations that will be used throughout the paper .

Nambu-Poisson manifolds are $n$-ary generalizations of Poisson manifolds introduced by Takhtajan \cite{takhtajan}.
\begin{defn}\label{nambu-poisson}
Let $M$ be a smooth manifold. A {\it Nambu-Poisson structure} of order $n$ on $M$ is a skew-symmetric $n$-multilinear bracket
\begin{align*}
 \{~, \ldots, ~\} : C^\infty (M) \times \stackrel{(n)}{\ldots} \times C^\infty (M) \rightarrow C^\infty (M)
\end{align*}
satisfying the following conditions:
\begin{enumerate}
 \item {\it Leibniz rule}:\\
  $\{f_1, \ldots, f_{n-1}, fg\} = f \{f_1, \ldots, f_{n-1}, g \} + \{f_1, \ldots, f_{n-1}, f \} g; $
 \item {\it Fundamental identity}:\\
 $\{f_1, \ldots ,f_{n-1}, \{g_1,\ldots, g_n\}\} = \sum_{k=1}^n  \{g_1,\ldots,g_{k-1},\{f_1,\ldots,f_{n-1},g_k\},\ldots, g_n\}; $
\end{enumerate}
for all $f_i, g_j, f, g \in C^\infty (M).$ A manifold together with a Nambu-Poisson structure of order $n$ is called a {\it Nambu-Poisson manifold of order $n$}. Thus the space of smooth functions with this bracket forms a Nambu-Poisson algebra. A Nambu-Poisson manifold of order $2$ is nothing but a Poisson manifold \cite{vais-book}. Since the bracket above is skew-symmetric and satisfies Leibniz rule, there exists an $n$-vector field $\Pi \in \mathcal{X}^{n}(M)$, such that
$$ \{f_1, \ldots, f_n\} = \Pi (df_1, \ldots, df_n),$$
for all $ f_1, \ldots, f_n \in C^\infty (M)$. Given any $(n-1)$ functions $ f_1, \ldots, f_{n-1} \in C^\infty (M)$, the {\it Hamiltonian vector field}
associated to these functions is denoted by $X_{f_1\ldots f_{n-1}}$ and is defined by
$$ X_{f_1\ldots f_{n-1}} (g) = \{f_1, \ldots ,f_{n-1}, g\}.$$
Note that the Fundamental identity, in terms of Hamiltonian vector fields, is equivalent to the condition
$$ [X_{f_1\ldots f_{n-1}}, X_{g_1\ldots g_{n-1}}] = \sum_{k=1}^{n-1} X_{g_1\ldots \{f_1, \ldots, f_{n-1},g_k\}\ldots g_{n-1}}.$$
for all $f_1, \ldots, f_{n-1}, g_1, \ldots, g_{n-1} \in C^\infty(M).$ The Fundamental identity can also be rephrased as
$$ \mathcal{L}_{X_{f_1\ldots f_{n-1}}} \Pi = 0,$$
for all $f_1, \ldots, f_{n-1} \in C^\infty(M),$  which shows that every Hamiltonian vector field preserves the Nambu tensor. A Nambu-Poisson manifold is often denoted by $(M, \{~, \ldots, ~\})$ or simply by $(M, \Pi).$
\end{defn}

\begin{exam}
\begin{enumerate}
\item Let $M$ be an orientable manifold of dimension $n$ and $\nu$ be a volume form on $M$. Define an $n$-bracket $\{~, \ldots, ~\}$ on
$C^\infty(M)$ by the following identity
$$ df_1 \wedge \cdots \wedge df_n = \{f_1, \ldots, f_n\} \nu.$$
Then $\{~, \ldots, ~\}$ defines a Nambu-Poisson structure of order $n$ on $M$. Let $\Pi_\nu \in \Gamma{\bigwedge^n TM}$ denotes the associated Nambu-Poisson
tensor. If $\Pi \in \Gamma{\bigwedge^n TM}$ is any Nambu-Poisson structure of order $n$  such that $\Pi \neq 0$ at every point, then there
exists a volume form $\nu'$ on $M$ such that $\Pi = \Pi_{\nu'}$. If $M = \mathbb{R}^n$ and $\nu = dx_1 \wedge \cdots \wedge dx_n$ is the standard volume form, then one recovers the Nambu structure on $\mathbb{R}^n$ originally discussed by Y. Nambu \cite{nambu}.
\item Let $\Pi$ be any $n$-vector field on an oriented manifold $M$ of dimension $n$. Then $\Pi$ defines a Nambu structure of order $n$ on $M$ (see \cite{ibanez1}).
\item Let $M$ be a manifold of dimension $m$ and $X_1, \ldots, X_n$ be linearly independent vector fields such that $[X_i, X_j] = 0$ for all
$i, j = 1, \ldots, n$. Then the $n$-vector field $\Pi = X_1 \wedge \cdots \wedge X_n$ defines a Nambu structure of order $n$.
\item Let $(M, \{~, \ldots, ~\})$ be a Nambu-Poisson manifold of order $n$. Suppose $k \leqslant n-2$ and $F_1, \ldots, F_k \in C^\infty(M)$ be any fixed functions on $M$. Define a $(n-k)$-bracket
$\{~, \ldots, ~\}'$ on $C^\infty(M)$ by
\begin{align*}
 \{f_1, \ldots, f_{n-k}\}' = \{F_1, \ldots, F_k, f_1, \ldots, f_{n-k}\}
\end{align*}
for $f_1, \ldots, f_{n-k} \in C^\infty(M)$. Then $\{~, \ldots, ~\}'$ defines a Nambu-structure of order $(n-k)$ on $M$. This Nambu structure
is called the {\it subordinate Nambu structure} of $(M, \{~, \ldots, ~\})$ with subordinate function $F_1, \ldots, F_k$.
 \item If $\Pi_i$ is a Nambu structure of order $n_i$ on a manifold $M_i$ $(i=1,2)$, then $\Pi = \Pi_1 \wedge \Pi_2$ is a Nambu structure of order $n_1 + n_2$ on $M_1 \times M_2$ \cite{duf-zung}.
\end{enumerate}
More examples of Nambu structures can be found in \cite{ibanez2, vais}.
\end{exam}


Let $(M, \Pi)$ be a Nambu-Poisson manifold of order $n$. For each $m \in M$, let $\mathcal{D}_mM \subset T_mM$ be the subspace of the tangent space at $m$ generated by all Hamiltonian vector fields at $m$. Since the Lie bracket of two Hamiltonians is again a Hamiltonian, therefore $\mathcal{D}$ defines a {\it (singular) integrable distribution} whose leaves are either $n$-dimensional submanifolds endowed with a volume form or just singletons \cite{duf-zung}.

\begin{defn}
Let $(M, \Pi_M)$ and $(N, \Pi_N)$ be two manifolds with $n$-vector fields. A smooth map $\phi: M \rightarrow N$ is called {\it $(\Pi_M, \Pi_N)$-map} if the induced brackets on functions satisfies:
\begin{align*}
 \{\phi^*f_1, \ldots, \phi^*f_n\}_M = \phi^* \{f_1, \ldots, f_n\}_N
\end{align*}
for all
$f_1, \ldots, f_n \in C^\infty(N),$ or equivalently, $\phi_* \Pi_M = \Pi_N.$ The map $\phi$ is called an {\it anti $(\Pi_M, \Pi_N)$-map} if 
\begin{align*}
 \{\phi^*f_1, \ldots, \phi^*f_n\}_M = (-1)^{n-1} \phi^* \{f_1, \ldots, f_n\}_N
\end{align*}
for all $f_1, \ldots, f_n \in C^\infty(N)$. A $(\Pi_M, \Pi_N)$-map $\phi : (M, \Pi_M) \longrightarrow (N, \Pi_N)$ between  Nambu-Poisson manifolds of the same order $n$ is called a {\it Nambu-Poisson map} or a $N$-$P$-map.
\end{defn}
\begin{remark}
The condition for a $(\Pi_M, \Pi_N)$-map can also be expressed in terms of the induced bundle maps as
\begin{align*}
\Pi^{\sharp}_{N, {\phi(m)}} = T_m \phi \circ \Pi^{\sharp}_{M, {m}} \circ T^*_m \phi ~ ~~\mbox{~~ for each}~~ m \in M,
\end{align*}
where $\Pi^{\sharp}_M :\bigwedge^{n-1}T^*M \rightarrow TM$ is the induced bundle map and is given by
\begin{align*}
 \langle \beta, \Pi^{\sharp}_M (\alpha_1 \wedge \cdots \wedge \alpha_{n-1}) \rangle = \Pi_M (\alpha_1, \ldots, \alpha_{n-1}, \beta)
\end{align*}
for all $\alpha_1, \ldots, \alpha_{n-1}, \beta \in T_x^* M,$ $x \in M$.
\end{remark}

\begin{defn}
 A {\it Lie groupoid} over a smooth manifold $M$ is a smooth manifold $G$ together with the following structure maps:
\begin{enumerate}
 \item two surjective submersions $\alpha, \beta : G \rightarrow M $, called the {\it source} map and the {\it target} map respectively;
 \item a smooth {\it partial multiplication} map
$$ G_{(2)} = \{(g,h) \in G \times G | \beta(g) = \alpha (h) \} \rightarrow G , ~~ (g,h)\mapsto gh ;$$
 \item a smooth {\it unit} map $\epsilon : M \rightarrow G$, $x \mapsto \epsilon_x ;$
 \item and a smooth {\it inverse} map $i : G \rightarrow G$, $g \mapsto g^{-1}$ with $\alpha (g^{-1}) = \beta (g)$ and $\beta (g^{-1}) = \alpha (g)$
\end{enumerate}
such that, the following conditions are satisfied\\
$\hspace*{1.5cm}$(i) $\alpha (gh) = \alpha (g)$ and $\beta(gh) = \beta (h)$;\\
$\hspace*{1.5cm}$(ii) $(gh)k = g(hk),$ whenever the multiplications make sense;\\
$\hspace*{1.5cm}$(iii) $\alpha(\epsilon_x) = \beta (\epsilon_x) = x$, $\forall x \in M$;\\
$\hspace*{1.5cm}$(iv) $\epsilon_{\alpha(g)} g = g$ and $ g \epsilon_{\beta(g)} = g$, $\forall g \in G$;\\
$\hspace*{1.5cm}$(v) $g g^{-1} = \epsilon_{\alpha(g)}$ and $g^{-1} g = \epsilon_{\beta(g)}$, $\forall g \in G$.
\end{defn}
A Lie groupoid $G$ over $M$ is denoted by $G \rightrightarrows M$ when all the structure maps are understood.

\begin{remark}
Note that the smooth structure on $G_{(2)}$ comes from the fact that 
$$G_{(2)} = (\beta \times \alpha)^{-1}(\Delta_M),$$ where $\beta \times \alpha : G \times G \rightarrow M \times M,~~(g,h) \mapsto (\beta(g), \alpha (h))$
and $ \Delta_M = \{(m,m)| m \in M\} \subset M \times M$ is the diagonal submanifold of $M \times M.$ Then these conditions imply that the inverse
map $i : G \rightarrow G,$ $g \mapsto g^{-1}$ is also smooth \cite{mac-book}. Moreover, $\alpha$-fibers and $\beta$-fibers are submanifolds of $G$ as both $\alpha$ and $\beta$ are surjective submersions.
\end{remark}

\begin{defn}
 Given a Lie groupoid $G \rightrightarrows M$, define an equivalence relation $'\sim'$ on $M$ by the following: two points $x, y \in M$ are said to be
equivalent, written as $x \sim y$, if there exists an element $g \in G$ such that $\alpha(g) = x$, $\beta(g) = y$. The quotient $M/\sim$ is called the {\it orbit set} of $G$.
\end{defn}

\begin{defn}
 Given two Lie groupoid $G_1 \rightrightarrows M_1$ and $G_2 \rightrightarrows M_2$, a {\it morphism} between Lie groupoids is a pair $(F, f)$ of smooth maps
$F : G_1 \rightarrow G_2$ and $f: M_1 \rightarrow M_2$ which commute with all the structure maps of $G_1$ and $G_2$. In other words,
$$ \alpha_2 \circ F = f \circ \alpha_1, ~ ~ \beta_2 \circ F = f \circ \beta_1, ~ ~ \text{and} ~ ~ F(g_1 h_1) = F(g_1)F(h_1)$$
for all $(g_1,h_1) \in (G_1)_{(2)}.$
\end{defn}

\begin{defn}
 Let $G \rightrightarrows M$ be a Lie groupoid. A {\it Lie subgroupoid} of it is a Lie groupoid $ H \rightrightarrows N$ together with injective immersions $i : H \rightarrow G$ and
$i_0 : N \rightarrow M$ such that $(i, i_0)$ is a Lie groupoid morphism.
\end{defn}

\begin{defn}
 Let $G \rightrightarrows M$ be a Lie groupoid. A submanifold $\mathcal{K}$ of $G$ is called a {\it bisection} of the Lie groupoid, if $\alpha|_{\mathcal{K}} : \mathcal{K} \rightarrow M$ and
$\beta|_{\mathcal{K}} : \mathcal{K} \rightarrow M$ are diffeomorphisms.
\end{defn}
 The existence of local bisections through any point $g \in G$ is always guaranted. The space of bisections $\mathcal{B}(G)$ form an infinite dimensional (Fr\'{e}chet) Lie group under the multiplication of subsets induced from the partial multiplication of $G$. Note that the left (right) multiplication is defined only on $\alpha$-fibers ($\beta$-fibers), therefore, we can not define a diffeomorphism of $G$ using left (right) multiplication by an element, like a Lie group. However we can do so by using bisection instead of an element.
Given a bisection $\mathcal{K} \in \mathcal{B}(G)$, let
$l_{\mathcal{K}}$ and $r_{\mathcal{K}}$ be the diffeomorphisms on $G$ defined by
$$ l_{\mathcal{K}} (h) = gh, \hspace{0.15cm} \text{where} \hspace{0.15cm} g \in \mathcal{K} \hspace{0.15cm} \text{is the unique element such that} \hspace{0.15cm} \beta(g) = \alpha(h)$$
and
$$ r_{\mathcal{K}} (h) = hg', \hspace{0.15cm} \text{where} \hspace{0.15cm} g' \in \mathcal{K} \hspace{0.15cm} \text{is the unique element such that} \hspace{0.15cm} \alpha(g') = \beta(h).$$

\begin{remark}
 Suppose $\mathcal{K}$ is any (local) bisection of $G$ through $g \in G$. Then the restriction of the map
$l_{\mathcal{K}}$ to $\alpha^{-1} (\beta(g))$ is the left translation $l_g$ by $g$:
$$ l_g : \alpha^{-1}(\beta(g)) \rightarrow \alpha^{-1}(\alpha(g)), ~ h \mapsto gh .$$
\end{remark}

Then we have the following result \cite{xu}.

\begin{prop}\label{left-invariant}
 Let $G \rightrightarrows M$ be a Lie groupoid and $P$ be an $n$-vector field on $G$. Suppose for any $g \in G$ with $\beta(g) = u$, $P$ satisfies $P(g) = (l_{\mathcal{G}})_* P(\epsilon_u)$, where $\mathcal{G} \in \mathcal{B}(G)$ is any arbitrary bisection through the point $g$. Then $P$ is left invariant.
\end{prop}

\begin{defn}
 A {\it Lie algebroid} $(A, [~, ~], a)$ over a smooth manifold M is a smooth vector bundle $A$ over
M together with a Lie algebra structure $[~, ~]$ on the space $\Gamma{A}$ of the smooth sections of $A$ and a bundle map $a : A \rightarrow T M $ , called the {\it anchor}, such that
\begin{enumerate}
 \item the induced map $ a : \Gamma{A} \rightarrow \mathcal{X}^1(M) $ is a Lie algebra homomorphism, where $\mathcal{X}^1(M)$ is the usual Lie algebra of vector fields on $M$.
 \item For any $ X, Y \in \Gamma{A} $ and $f \in C^\infty (M)$, we have
$$ [X, f Y ] = f [X, Y ] + (a(X)f )Y. $$
\end{enumerate}
\end{defn}
We may denote a Lie algebroid simply by $A$, when all the structures are understood. Any Lie algebra is a Lie algebroid over a point with zero anchor. The tangent bundle of any smooth manifold is a Lie algebroid with usual Lie bracket of vector fields and identity as anchor.\\
{\bf Lie algebroid of a Lie groupoid.}
Given a Lie groupoid $G \rightrightarrows M$, its Lie algebroid consists of the vector bundle $AG \rightarrow M$ whose fiber at $x \in M$ coincides with the tangent space
at the unit element $\epsilon_x$ of the $\alpha$-fiber at $x$. Then the space of sections of $AG$ can be identified with the left invariant vector fields
\begin{align*}
 \mathcal{X}^1_{\text{inv}}(G) = \{ X \in \Gamma (T^\alpha G) = \Gamma (\text{ker} (d\alpha))| X_{gh} = (l_g)_* X_h, \forall (g, h) \in  G_{(2)} \}
\end{align*}
on $G$. Since the space of left invariant vector fields on $G$
is closed under the Lie bracket, therefore it defines a Lie bracket on $\Gamma AG$. The anchor $a$ of $AG$ is defined to be the differential of the target map $\beta$ restricted to $AG$.

Let $AG $ be the Lie algebroid of the Lie groupoid $G \rightrightarrows M$. Given any $X \in \Gamma AG$, let $\overleftarrow{X}$ be the corresponding left invariant vector field on $G$. Then there exists an $\epsilon > 0$ and a $1$-parameter family of transformations $\phi_t$ $(|t| < \epsilon)$, generated by $\overleftarrow{X}$ (\cite{mac-book}). Suppose each
$\phi_t$ is defined on all of $M$, where $M$ is identified with a closed embedded submanifold of $G$ via the unit map. We denote the image of $M$ via $\phi_t$
by  exp $tX$. Then exp $tX$ is a bisection of the groupoid (for all $|t| < \epsilon$) and satisfies $1$-parameter group like conditions, namely
\begin{align*}
 \text{exp} (t+s) X =  \text{exp} \hspace*{0.1cm} tX \cdot \text{exp} \hspace*{0.1cm}sX, \hspace*{1cm} \text {whenever} \hspace*{0.5cm} |t| , |s| , |t+s| < \epsilon,
\end{align*}
where on the right hand side, we used the multiplication of bisections.

\section{Coisotropic submanifolds}
Let $M$ be a manifold and $\Pi \in \mathcal{X}^n(M)$ be a $n$-vector field on $M$. Let
\begin{align*}
 \Pi^{\sharp} : {\bigwedge}^{n-1}T^*M \rightarrow TM
\end{align*}
be the induced bundle map given by
\begin{align*}
 \langle \beta, \Pi^{\sharp} (\alpha_1 \wedge \cdots \wedge \alpha_{n-1}) \rangle = \Pi (\alpha_1, \ldots, \alpha_{n-1}, \beta)
\end{align*}
for all $\alpha_1, \ldots, \alpha_{n-1}, \beta \in T_x^* M,$ $x \in M$.

We recall the following definition from \cite{pont-geng-xu}.
\begin{defn}
A submanifold $C \hookrightarrow M$ is said to be {\it coisotropic} with respect to $\Pi$,
if
\begin{align*}
 \Pi^{\sharp} ({\bigwedge}^{n-1} (TC)^0) \subset TC
\end{align*}
where $$(TC)^0_x = \{ \alpha \in  T_x^* M | \hspace*{0.1cm} \alpha (v) = 0, \forall v \in T_xC\},~~x \in C,$$
or equivalently,
\begin{align*}
 \Pi_x (\alpha_1, \ldots, \alpha_n) = 0, \forall \alpha_i \in (TC)^0_x, x \in C.
\end{align*}
\end{defn}

We have the following easy observation for coisotropic submanifolds of a Nambu-Poisson manifold.
\begin{prop}
 Let $(M, \Pi)$ be a Nambu-Poisson manifold of order $n$ and $C$ be a closed embedded submanifold of $M$. Let $\mathcal{I} (C) = \{ f \in C^\infty(M)\big| f|_C \equiv 0 \}$ denote the vanishing ideal of $C$. Then the followings are equivalent:
\begin{enumerate}
\item $C$ is a coisotropic submanifold;
\item $\mathcal{I} (C)$ is a Nambu-Poisson subalgebra;
\item for every $f_1, \ldots, f_{n-1} \in \mathcal{I} (C)$, the Hamiltonian vector field $X_{f_1...f_{n-1}}$ is tangent to $C$.
\end{enumerate}
\end{prop}
\begin{proof}
 (1) $\Rightarrow$ (2) Let $f_1, \ldots, f_n \in \mathcal{I} (C) $.
Then for any $x \in C$, $d_xf_i \in (TC)^0_x$, for all $i=1, \ldots, n.$ Now since $C$ is a coisotropic submanifold, we have
\begin{align*}
 \{f_1, \ldots, f_n\}(x) = \Pi_x (d_xf_1, \ldots, d_xf_n) = 0,~ \forall x \in C.
\end{align*}
Hence $\{f_1, \ldots, f_n\} \in \mathcal{I} (C)$. Therefore $\mathcal{I} (C)$ is a Nambu-Poisson subalgebra.

(2) $\Rightarrow$ (3) Let $f_1, \ldots, f_{n-1} \in \mathcal{I} (C)$ and $x \in C$.
Let $\alpha \in (TC)^0_x.$ Then there exists a function $g$ vanishing on $C$ such that $d_xg = \alpha$. Since $\mathcal{I} (C)$ is a Nambu-Poisson subalgebra, we have
\begin{align*}
 \{f_1, \ldots, f_{n-1}, g \}(x) = 0.
\end{align*}
Thus,
\begin{align*}
 X_{f_1\ldots f_{n-1}}\big|_x (\alpha) = X_{f_1\ldots f_{n-1}}\big|_x (d_xg) = \{f_1, \ldots, f_{n-1}, g \}(x) = 0
\end{align*}
and consequently, $X_{f_1\ldots f_{n-1}}$ is tangent to $C$.

(3) $\Rightarrow$ (1) Let $x \in C$ and $\alpha_1, \ldots, \alpha_n \in (TC)^0_x.$ Then there exist functions $f_1, \ldots, f_n \in \mathcal{I}(C)$ such that
$d_xf_i = \alpha_i,$ $\forall i = 1, \ldots, n.$ Therefore,
\begin{align*}
 \Pi_x (\alpha_1, \ldots, \alpha_n) = \Pi_x (d_xf_1, \ldots, d_xf_n) = X_{f_1...f_{n-1}} \big|_x (d_xf_n) = 0.
\end{align*}
Hence $C$ is a coisotropic submanifold of $M$.
\end{proof}

\begin{prop}
 Let $(M, \Pi_M)$ and $(N, \Pi_N)$ be two Nambu-Poisson manifolds of same order $n$ and $C \hookrightarrow N$ be a coisotropic submanifold of $N$ with respect to $\Pi_N$. If $\phi: M \rightarrow N$
is a Nambu-Poisson map transverse to $C$, then $\phi^{-1}(C)$ is a coisotropic submanifold of $M$ with respect to $\Pi_M$ (the result holds true for manifolds with $n$-vector fields such that
$\phi_* \Pi_M = \Pi_N$).
\end{prop}

\begin{proof}
Since $\phi$ is transverse to $C$, therefore $\phi^{-1}(C)$ is a submanifold of $M$. Moreover
\begin{align*}
 T(\phi^{-1}(C)) = (T \phi)^{-1} TC.
\end{align*}
Therefore $T(\phi^{-1}(C))^0 = (T \phi)^{*} (TC)^0$. Observe that
\begin{align*}
 T\phi (\Pi_M^{\sharp} \big( {\bigwedge}^{n-1} T(\phi^{-1}(C))^0  \big)) =& T\phi ( \Pi_M^{\sharp} \big((T \phi)^{*} {\bigwedge}^{n-1} (TC)^0 \big))\\
=& \Pi_N^{\sharp} ({\bigwedge}^{n-1} (TC)^0)\\
\subseteq & TC.
\end{align*}
Thus,
\begin{align*}
 \Pi_M^{\sharp} \big( {\bigwedge}^{n-1} T(\phi^{-1}(C))^0  \big) \subseteq (T\phi)^{-1} TC = T(\phi^{-1}(C))
\end{align*}
and hence $\phi^{-1}(C)$ is coisotropic with respect to $\Pi_M.$
\end{proof}

\begin{prop}\label{image-coiso}
 Let $\phi: (M, \Pi_M) \rightarrow (N, \Pi_N)$ be a Nambu Poisson map between two Nambu-Poisson manifolds $(M, \Pi_M)$ and $(N, \Pi_N)$ and $C\hookrightarrow M$ be a coisotropic submanifold
of $M$. Assume that $\phi(C)$ is a submanifold of $N$. Then $\phi(C)$ is a coisotropic submanifold of $N$ (the result holds true for manifolds with $n$-vector fields such that $\phi_{*}\Pi_M = \Pi_N$).
\end{prop}

\begin{proof}
We have $T(\phi(C)) \supseteq T\phi (TC)$ and $(T\phi)^* (T(\phi(C)))^0 \subseteq (TC)^0$. Therefore,
\begin{align*}
 \Pi_N^{\sharp} ({\bigwedge}^{n-1} T(\phi(C))^0) =& T\phi (\Pi_M^{\sharp} ((T\phi)^* {\bigwedge}^{n-1} T(\phi(C))^0)) \hspace*{1cm}(\text{since ~} \phi \text{~ is a N-P map})\\
\subseteq & T\phi (\Pi_M^{\sharp} ({\bigwedge}^{n-1} (TC)^0))\\
\subseteq & T\phi (TC) \hspace*{1cm}(\text{since~} C \hookrightarrow M \text{~ is coisotropic})\\
\subseteq & T(\phi(C))
\end{align*}
which shows that $\phi(C)$ is a coisotropic submanifold of $N$.
\end{proof}

Using the terminology of coisotropic submanifold with respect to any multivector field allows us to extend the results of Weinstein \cite{wein} from Poisson bivector field to Nambu-Poisson tensor or more generally to any multivector field.

\begin{prop}\label{nambu map-coiso}
 Let $(M, \Pi_M)$ and $(N, \Pi_N)$ be two manifolds with $n$-vector fields and $\phi: M \rightarrow N$ be a smooth map. Then $\phi$ is a $(\Pi_M, \Pi_N)$-map, that is $\phi_* \Pi_M = \Pi_N$ if and only if its graph
\begin{align*}
 \text{Gr}(\phi) = \{(m, \phi(m))| m \in M \}
\end{align*}
is a coisotropic submanifold of $M \times N$ with respect to $\Pi_M \oplus (-1)^{n-1} {\Pi}_N$.
\end{prop}

\begin{proof}
 Let $C = \text{Gr}(\phi) \subset M \times N$. Then $C$ is a closed embedded submanifold of $M \times N.$
Note that, a tangent vector to the graph consist of a pair $(v_m, (T\phi)(v_m))$, where $m \in M$, $v_m \in T_mM$. Therefore,
$(TC)^0$ consists of a pair of covectors $(-(T\phi)^*\psi, \psi)$, where $\psi \in T_{\phi(m)}^*N.$
Therefore, Gr($\phi)$ is a coisotropic submanifold of $M \times N$ with respect to $\Pi_M \oplus (-1)^{n-1} {\Pi}_N$ if and only if
$(\Pi_M^{\sharp} \times (-1)^{n-1} \Pi_N^{\sharp})$ maps $(-(T\phi)^*\psi_1, \psi_1) \wedge \cdots \wedge (-(T\phi)^*\psi_{n-1}, \psi_{n-1}) $ into $TC$,
for all $\psi_1, \ldots, \psi_{n-1} \in T^*_{\phi(m)}N$ and $m \in M$. In other words,
\begin{align*}
(T\phi) \bigg(\Pi_M^{\sharp} (- (T\phi)^*\psi_1, \ldots, - (T\phi)^*\psi_{n-1})\bigg) = (-1)^{n-1} {\Pi}_N^{\sharp} (\psi_1, \ldots, \psi_{n-1})
\end{align*}
that is,
\begin{align*}
(T\phi) \bigg(\Pi_M^{\sharp} ( (T\phi)^*\psi_1, \ldots,  (T\phi)^*\psi_{n-1})\bigg) =  {\Pi}_N^{\sharp} (\psi_1, \ldots, \psi_{n-1}).
\end{align*}
This is equivalent to the condition that $\phi$ is a $(\Pi_M, \Pi_N)$-map.
\end{proof}

\begin{defn}
 Let $(M, \Pi_M)$ be a Nambu-Poisson manifold of order $n$ and $\phi: M \rightarrow N$ be a smooth surjective map. If there exist a Nambu-Poisson structure $\Pi_N$ (of order $n$) on $N$
which makes $\phi$ into a Nambu-Poisson map, then $\Pi_N$ is called the Nambu-Poisson structure {\it coinduced} by the mapping $\phi$.
\end{defn}

The following is a characterization of coinduced Nambu-Poisson structure.
 
\begin{prop}\label{coinduced}
 Let $(M, \Pi_M)$ be a Nambu-Poisson manifold of order $n$ and $\phi: M \rightarrow N$ be a smooth surjective map from $M$ to some manifold $N$. Then $N$ has a Nambu-Poisson structure coinduced by $\phi$ if and only if for all $f_1, \ldots, f_{n} \in C^\infty(N)$, the function
$\{\phi^*f_1, \ldots, \phi^*f_n\}_M$ is constant along the fibers of $\phi$.
\end{prop}

\begin{proof}
 Let $f_1, \ldots, f_n \in C^\infty(N)$. If the function $\{ \phi^*f_1, \ldots, \phi^*f_n \}_M$ is constant along the $\phi$-fibers, then there exists a function on $N$, which we denote by $\{f_1, \ldots, f_n\}_N$ such that
$\{\phi^*f_1, \ldots, \phi^*f_n\}_M = \phi^* \{f_1, \ldots, f_n\}_N $. Clearly this bracket defines a coinduced Nambu-Poisson structure on $N$.

Conversely, suppose that there is a Nambu-Poisson bracket $\{~, \ldots, ~\}_N$ on $N$ coinduced by $\phi$. Then for any $y \in N,$
\begin{align*}
\{\phi^*f_1, \ldots, \phi^*f_n\}_M(\phi^{-1}\{y\}) =& (\phi^* \{f_1, \ldots, f_n\}_N )(\phi^{-1}y)\\
=& \{f_1, \ldots, f_n\}_N (y)
\end{align*}
proving $\{\phi^*f_1, \ldots, \phi^*f_n\}_M$ is constant along the $\phi$-fibers.
\end{proof}

\begin{remark}\label{rem-coinduced}
 Let $(M, \Pi_M)$ be a manifold with an $n$-vector field and $\phi: M \rightarrow N$ be a smooth map. Then there exists an $n$-vector field
$\Pi_N$ on $N$ such that $\phi$ is a $(\Pi_M, \Pi_N)$-map if and only if for all $f_1, \ldots, f_n \in C^\infty(N)$, the function
$\{\phi^*f_1, \ldots, \phi^*f_n\}_M$ is constant along the fibers of $\phi$.
\end{remark}

\begin{prop}
Let $(M, \Pi_M)$ be a Nambu-Poisson manifold and $\phi : M \rightarrow N$ be a surjective submersion with connected fibers. Let $\text{ker} \hspace*{0.1cm} \phi_*(m)$ is spanned by local Hamiltonian vector fields (that is, $\text{ker} \hspace*{0.1cm} \phi_*(m) \subset \mathcal{D}_mM $), for all $m \in M$. Then 
$N$ has a Nambu-Poisson structure  coinduced by $\phi.$
\end{prop}

\begin{proof}
 Since $\phi$ is a submersion, the fibers of $\phi$ are submanifolds of $M$. Then for $y \in N,$  $\phi^{-1}(\{y\}) = C$ is a submanifold of $M$. Let $g_1, \ldots, g_{n-1}$ be locally defined functions on $M$ such that $X_{g_1...g_{n-1}} \in \text{ker} \hspace*{0.1cm} \phi_{*}$. Let $f_1, \ldots, f_n \in C^\infty(N).$ To prove that $\{\phi^*f_1, \ldots, \phi^*f_n\}$
is constant on the fibers, it is enough to prove that
\begin{align*}
 X_{g_1...g_{n-1}} \{\phi^*f_1, \ldots, \phi^*f_n\} = 0.
\end{align*}
Note that
\begin{align*}
 X_{g_1...g_{n-1}} \{\phi^*f_1, \ldots, \phi^*f_n\} = \sum_{k=1}^n \{\phi^*f_1, \ldots, X_{g_1...g_{n-1}} (\phi^*f_k), \ldots, \phi^*f_n\}
\end{align*}
and the functions $\phi^*f_i$ are constant along the fibers. Hence by the Proposition \ref{coinduced}, there exists a coinduced Nambu-Poisson structure on $N.$
\end{proof}

\begin{prop}\label{coinduced-coiso}
Let $(M, \Pi_M)$ be a  manifold with an $n$-vector field and $\phi: M \rightarrow N$ be a surjective submersion. Then $N$ has an (unique) $n$-vector field $\Pi_N$ such that
$\phi$ is a $(\Pi_M, \Pi_N)$-map if and only if
 $R(\phi) = \{(x,y) \in M \times M|\phi(x)=\phi(y)\}$ is a coisotropic submanifold of $M \times M$ with respect to $\Pi_M \oplus (-1)^{n-1}{\Pi}_M$.
\end{prop}

\begin{proof}
Note that $R(\phi) = (\phi \times \phi)^{-1} (\Delta_N)$, where
$\Delta_N$ is the diagonal of $N \times N$. Since $\phi$ is surjective submersion $R(\phi)$ is a submanifold of $M \times M$.
Moreover, for $(x,y) \in R(\phi)$
\begin{align*}
 T_{(x,y)} (R(\phi)) = \{(X, Y)\in T_xM \times T_yM | (T\phi)_x (X) = (T\phi)_y (Y)\}.
\end{align*}
Therefore, $T(R(\phi))^0$ consists of covectors $(-(T\phi)_x^*\psi, (T\phi)_y^*\psi)$, where $\psi \in T_{\phi(x)}^*N$.

Thus, $R(\phi)$ be a coisotropic submanifold of $M \times M$ with respect to $\Pi_M \oplus (-1)^{n-1}{\Pi}_M$
if and only if for all $\psi_1, \ldots, \psi_{n-1} \in T_{\phi(x)}^*N$ and $(x, y) \in R(\phi),$ $\Pi_M^{\sharp} \oplus (-1)^{n-1}{\Pi}_M^{\sharp}$ maps 
$$(-(T\phi)_x^*\psi_1, (T\phi)_y^*\psi_1) \wedge \cdots \wedge (-(T\phi)_x^*\psi_{n-1}, (T\phi)_y^*\psi_{n-1})$$ into $T(R(\phi)).$ 
That is
$$(T\phi)_x \Pi_M^{\sharp} (-(T\phi)_x^*\psi_1, \ldots, -(T\phi)_x^*\psi_{n-1} )= (-1)^{n-1} (T\phi)_y \Pi_M^{\sharp} ((T\phi)_y^*\psi_1, \ldots, (T\phi)_y^*\psi_{n-1}),$$
or equivalently,
\begin{align} \label{firsteqn}
 (T\phi)_x \Pi_M^{\sharp} ((T\phi)_x^*\psi_1, \ldots, (T\phi)_x^*\psi_{n-1} ) =  (T\phi)_y \Pi_M^{\sharp} ((T\phi)_y^*\psi_1, \ldots, (T\phi)_y^*\psi_{n-1})
\end{align}
holds. Let $f_1, \ldots ,f_n \in C^\infty(N)$ and $x \in M.$ Then
\begin{align*}
\{\phi^*f_1, \ldots , \phi^*f_n\}_M (x) =& \langle \Pi_M^{\sharp} (d_x (\phi^*f_1) \wedge \cdots \wedge d_x (\phi^*f_{n-1})), d_x (\phi^*f_n) \rangle\\
 =& \langle \Pi_M^{\sharp} \big((T\phi)_x^* \psi_1 \wedge \cdots \wedge (T\phi)_x^* \psi_{n-1}\big), (T\phi)_x^* \psi_n \rangle\\
 =& \langle (T\phi)_x \Pi_M^{\sharp} \big((T\phi)_x^* \psi_1 \wedge \cdots \wedge (T\phi)_x^* \psi_{n-1}\big), \psi_n \rangle
\end{align*}
where $\psi_i = d_{\phi(x)}f_i = d_{\phi(y)}f_i\in T_{\phi(x)}^*N$, for all $1 \leqslant i \leqslant n.$ It follows from the Equation (\ref{firsteqn}) that the function  $\{\phi^*f_1, \ldots , \phi^*f_n\}_M$ is constant along the $\phi$-fibers if and only if $R(\phi)$ is a coisotropic submanifold of $M \times M$ with respect to $\Pi_M^{\sharp} \oplus (-1)^{n-1}{\Pi}_M^{\sharp}$. Hence the result follows by the Remark \ref{rem-coinduced}. The uniqueness follows from the surjectivity of $\phi.$
\end{proof}

\section{Nambu-Lie groupoids}
In this section, we recall the definition of multiplicative multivector fields on Lie groupoid (\cite{pont-geng-xu}) and define Nambu-Lie groupoid (of order $n$) as a Lie groupoid with a multiplicative
$n$-vector field which is also a Nambu-Poisson tensor.

\begin{defn}
Let $G \rightrightarrows M$ be a Lie groupoid and $\Pi \in \mathcal{X}^n(G)$ be an $n$-vector field on $G$. Then $\Pi$ is called {\it multiplicative} if the graph of the groupoid multiplication
\begin{align*}
 \{ (g, h, gh) \in G \times G \times G | ~ \beta(g) = \alpha(h)\}
\end{align*}
is a coisotropic submanifold of $G \times G \times G$ with respect to $\Pi \oplus \Pi \oplus (-1)^{n-1} \Pi$.
\end{defn}
 
Then we have the following characterization of multiplicative multivector fields \cite{pont-geng-xu}:
\begin{thm} \label{multiplicative}
 Let $G \rightrightarrows M$ be a Lie groupoid and $\Pi \in \mathcal{X}^n (G)$ be an $n$-vector field on $G$. Then $\Pi$ is multiplicative if and only if the following conditions are satisfied.
\begin{enumerate}
 \item $\Pi$ is an affine tensor. In other words
\begin{align*}
 \Pi(gh) = (r_{\mathcal{H}})_* \Pi(g) + (l_{\mathcal{G}})_* \Pi(h)  - (r_{\mathcal{H}})_* (l_{\mathcal{G}})_* \Pi(u)
\end{align*}
where $u = \beta(g) = \alpha(h)$ and $\mathcal{G}, \mathcal{H}$ are (local) bisections through the points $g, h$ respectively.
 \item $M$ is a coisotropic submanifold of $G$ with respect to $\Pi$.
 \item For all $g \in G$, $\alpha_* \Pi(g)$ and $\beta_* \Pi(g)$ depend only on the base points $\alpha(g)$ and $\beta(g)$ respectively.
 \item For all $f, f' \in C^\infty(M)$, the $(n-2)$-vector field $\iota_{d(\alpha^*f) \wedge d(\beta^*f')} \Pi$ is zero. In other words,
\begin{align*}
 \{~, \ldots, \alpha^*f, \beta^*f' \} = 0.
\end{align*}
 \item For all $f_1, \ldots, f_k \in C^\infty(M)$, $\iota_{d(\beta^*f_1) \wedge \cdots \wedge d(\beta^*f_k)} \Pi$ is a left invariant $(n-k)$-vector field on $G$, $1 \leqslant k < n.$
\end{enumerate}
\end{thm}

\begin{remark}\label{lie-gp-multiplicative}
 Suppose $G$ be a Lie group considered as a Lie groupoid over a point. Then the conditions {\it (3)} - {\it (5)} of the Theorem \ref{multiplicative} are satisfied automatically. The condition {\it (2)} implies that $\Pi (e) = 0$ (where $e$ is the identity element of the group), which together with condition {\it (1)} implies
that $\Pi$ satisfies the usual multiplicativity condition
$$ \Pi (gh) = (r_h)_* \Pi(g) + (l_g)_* \Pi(h) .$$
\end{remark}

\begin{defn}\label{nambu-lie groupoid} 
 A {\it Nambu-Lie groupoid of order $n$} is a Lie groupoid $G \rightrightarrows M$ with a multiplicative Nambu tensor $\Pi \in \mathcal{X}^n(G)$ of order $n.$ 
\end{defn}
A Nambu Lie groupoid (of order $n$) will be denoted
by $(G \rightrightarrows M, \Pi) $.

\begin{exam}\label{exam-nlg}
\begin{enumerate}
\item Poisson groupoids \cite{wein} are examples of Nambu-Lie groupoids with $n=2.$
\item Any Lie groupoid with zero Nambu structure is a Nambu-Lie groupoid.
\item Let $(G, \Pi)$ be a Nambu-Lie group (of order $n$) \cite{vais}. Thus $G$ is a Lie group equipped with  a Nambu structure $\Pi$ of order $n$ on $G$
such that
\begin{align*}
 \Pi(gh) = (r_h)_* \Pi(g) + (l_g)_* \Pi(h)
\end{align*}
for all $g, h \in G$. Note that the right hand side of the above equality is equal to $m_* (\Pi(g), \Pi(h)),$ where $m_* : \bigwedge^n T_{(g,h)} (G \times G) \rightarrow \bigwedge^n T_{gh}G$
is the map induced by the multiplication map $m : G \times G \rightarrow G$. Therefore,
\begin{align*}
 \Pi(gh) = m_* (\Pi(g), \Pi(h)).
\end{align*}
Thus, the group multiplication map $m : G \times G \rightarrow G$ is a $(\Pi \oplus \Pi, \Pi)$-map. Therefore, by the Proposition \ref{nambu map-coiso}, the graph of the group multiplication map is a coisotropic submanifold of $G \times G \times G$ with respect to $\Pi \oplus \Pi \oplus (-1)^{n-1} \Pi$. Hence $(G, \Pi)$ is a Nambu-Lie groupoid over a point.
Conversely, if $(G, \Pi)$ is a Nambu-Lie groupoid over a point, then the group multiplication map $m : G \times G \rightarrow G$ is a 
$(\Pi \oplus \Pi, \Pi)$-map. Hence $(G, \Pi)$ is a Nambu-Lie group in the sense of \cite{vais}. One can also see the equivalence between Nambu-Lie groupoid
over a point and Nambu-Lie group by using Remark \ref{lie-gp-multiplicative}.
\end{enumerate}
\end{exam}

For a Poisson groupoid the following facts are well known \cite{wein}.
\begin{itemize}
\item The groupoid inversion map is a anti-Poisson map. 
\item The Poisson structure on the total space induces a Poisson structure on the base such that the source map is a Poisson map and the target map is a anti-Poisson map. 
\end{itemize}

In the next proposition we generalize the above facts to the Nambu-Poisson setting.

\begin{prop}\label{inverse-basenambu}
 Let $(G \rightrightarrows M , \Pi)$ be a Nambu-Lie groupoid. Then
\begin{enumerate} 
\item The inverse map $i : G \rightarrow G$, $g \mapsto g^{-1}$ is an anti-Nambu Poisson map.
\item There is a unique Nambu-Poisson structure on $M$ which we denote by $\Pi_M$ for which $\alpha$ is a Nambu-Poisson map and $\beta$ is an anti Nambu-Poisson map.
\end{enumerate}
\end{prop}

\begin{proof}
(1) It is proved in \cite{pont-geng-xu} that given a Lie groupoid $G \rightrightarrows M$ with multiplicative $n$-vector field $\Pi \in \mathcal{X}^n(G)$,
the groupoid inversion map $i : G \rightarrow G$ satisfies
\begin{align*} 
 i_* \Pi = (-1)^{n-1} \Pi.
\end{align*}
Hence the result follows as $\Pi$ is a Nambu tensor.

(2) Let $f_1, \ldots, f_n \in C^\infty(M)$ be any functions on $M$. Then for any $g \in G$, we have
\begin{align*}
\{\alpha^* f_1, \ldots, \alpha^* f_n \} (g) =& \Pi (g) (d_g (\alpha^* f_1), \ldots, d_g (\alpha^* f_n))\\
 =& \Pi(g) (\alpha^* (d_{\alpha(g)} f_1), \ldots, \alpha^* (d_{\alpha(g)} f_n))\\
 =& \alpha_* \Pi (g) (d_{\alpha(g)} f_1, \ldots, d_{\alpha(g)} f_n).
\end{align*}
Since $\alpha_* \Pi(g)$ depends only on the value of $\alpha(g)$, it follows that the function $\{ \alpha^* f_1, \ldots, \alpha^* f_n \}$ is constant on the $\alpha$-fibers. Therefore, by the Proposition \ref{coinduced}, there exists a Nambu-Poisson structure $\Pi_M$ with the induced bracket denoted by $\{ ~, \ldots, ~\}_M,$ on $M$
for which $\alpha$ is a Nambu-Poisson map. Since $\beta = \alpha \circ i$ and $i$ is anti Nambu-Poisson, therefore $\beta$
is an anti Nambu-Poisson map.
\end{proof}

\begin{remark}
Consider the map $(\alpha,\beta): G \rightarrow M \times M.$
Since we have $\alpha_{*} \Pi = \Pi_M$ and $\beta_{*}\Pi = (-1)^{n-1} \Pi_M$, using property {\it (4)} of the Theorem \ref{multiplicative} we obtain 
\begin{align*}
(\alpha,\beta)_{*} \Pi = \alpha_* \Pi \oplus \beta_* \Pi = \Pi_M \oplus (-1)^{n-1} \Pi_M.
\end{align*}
\end{remark}

\begin{prop}
 Let $(G \rightrightarrows M, \Pi)$ be a Nambu-Lie groupoid. If the orbit space $M/\sim$ is a smooth manifold, then $M/\sim$ carries a Nambu-Poisson structure such that the projection $q: M \rightarrow M/\sim $ is a Nambu-Poisson map.
\end{prop}

\begin{proof}
Let $\Pi_M$ be the induced Nambu structure on the base $M$. For the projection map $q: M \rightarrow M/\sim $, we have
\begin{align*}
 R(q) =& \{(x,y) \in M \times M| q(x) = q(y)\}\\
=& \{(\alpha(g), \beta(g))| g \in G\}\\
=& (\alpha,\beta)(G).
\end{align*}
Consider $G$ as a coisotropic submanifold of $G$ with respect to $\Pi$ and also consider the map $(\alpha, \beta) : G \rightarrow M \times M$. By the above Remark we have $(\alpha, \beta)_* \Pi = \Pi_M \oplus (-1)^{n-1} \Pi_M.$ Therefore, by the Proposition \ref{image-coiso},  $R(q) = (\alpha,\beta)(G)$ is a coisotropic submanifold of $M \times M$ with respect to $\Pi_M \oplus (-1)^{n-1}\Pi_M.$ Hence the result follows from the Proposition \ref{coinduced-coiso}.
\end{proof}

\section{Infinitesimal form of Nambu-Lie groupoid}

The aim of this section, is to  study the infinitesimal form of a
Nambu-Lie groupoid. We show that if $(G \rightrightarrows M, \Pi)$ is a Nambu-Lie groupoid of order $n$ with Lie algebroid $AG \rightarrow M$, then
$(AG, A^*G)$ forms a weak Lie-Filippov bialgebroid of order $n$ introduced in \cite{bas-bas-das-muk}. Before proceeding further, let us briefly recall from \cite{bas-bas-das-muk} the notion of a weak Lie-Filippov bialgebroid.

Lie bialgebroids are  generalization of both Poisson manifolds and Lie bialgebras. Recall that a Lie bialgebroid, introduced by Mackenzie and Xu \cite{mac-xu} is also the infinitesimal form of a Poisson groupoid. It is defined as a pair $(A, A^\ast)$ of Lie algebroids in duality, where the Lie bracket of $A$ satisfies the following compatibility condition expressed in terms of the differential $d_\ast$ on $\Gamma (\bigwedge^\bullet A)$ 
$$d_\ast[X, Y] = [d_\ast X, Y] + [X , d_\ast Y],$$ 
for all $X,~ Y \in \Gamma A.$ 

We note that if $M$ is a Poisson manifold then the Lie algebroid structures on $TM$ and $T^\ast M$ form a Lie bialgebroid. On the other hand, it is well known \cite{kosmann,mac-xu} that if $(A, A^\ast)$ is a Lie bialgebroid over a smooth manifold $M$ then there is a canonical Poisson structure on the base manifold $M$.

Thus it is natural to ask the following question which was posed in \cite{bas-bas-das-muk}:\\ 
{\it Does there exist some notion of bialgebroid associated to a Nambu-Poisson manifold of order $n > 2$ }? 

To answer this question, the authors \cite{bas-bas-das-muk} introduced the notion of weak Lie-Filippov bialgebroid.

It is well known \cite{gra-mar, vais} that for a Nambu-Poisson manifold $M$ of order $n \geqslant 2$, the space $\Omega^1(M)$ of $1$-forms admits an $n$-ary bracket, called {\it Nambu-form bracket}, such that the bracket satisfies almost all the properties of an $n$-Lie algebra (also known as Filippov algebra of order $n$) bracket \cite{fil} except that the Fundamental identity is satisfied only in a restricted sense as described below.

Let $(M, \{~, \ldots, ~\})$ be a Nambu-Poisson manifold of order $n$ with associated Nambu-Poisson tensor $\Pi$. Then one can define the Nambu form-bracket on the space of $1$-forms
\begin{align*}
 [~, \ldots, ~]: \Omega^1(M) \times  \cdots \times \Omega^1(M) \rightarrow \Omega^1(M)
\end{align*}
 by the following
\begin{align*}
 [\alpha_1, \ldots, \alpha_n] = \sum_{k=1}^n (-1)^{n-k} \mathcal{L}_{\Pi^{\sharp}(\alpha_1 \wedge\cdots \wedge \hat{\alpha}_k \wedge\cdots  \wedge \alpha_n)} \alpha_k - (n-1) d (\Pi(\alpha_1, \ldots, \alpha_n))
\end{align*}
\begin{align}\label{nambu-bracket}
\hspace*{1cm} = d (\Pi(\alpha_1, \ldots, \alpha_n)) + \sum_{k=1}^n (-1)^{n-k} \iota_{\Pi^{\sharp}(\alpha_1 \wedge\cdots \wedge \hat{\alpha}_k \wedge\cdots \wedge \alpha_n)} d\alpha_k
\end{align}
for $\alpha_i \in \Omega^1(M), i=1,\ldots ,n.$ Here $\hat{\alpha}_k$ in a monomial $\alpha_1 \wedge \cdots \wedge \hat{\alpha}_k \wedge \cdots \wedge \alpha_n$ means that the symbol $\alpha_k$ is missing in the monomial. The above bracket satisfies the following properties (\cite{vais}).
\begin{enumerate} 
\item The bracket is skew-symmetric.
\item $[df_1, \ldots, df_n] = d \{f_1, \ldots, f_n \}$.
\item $[\alpha_1,\ldots, \alpha_{n-1}, f \alpha_n] = f [\alpha_1,\ldots, \alpha_{n-1},  \alpha_n] + \Pi ^{\sharp} (\alpha_1 \wedge\cdots \wedge \alpha_{n-1})(f) \alpha_n$.
\item The bracket satisfies the Fundamental identity
\begin{align*}
 [\alpha_1, \ldots, \alpha_{n-1}, [\beta_1, \ldots, \beta_n]] = \sum_{k=1}^n [\beta_1, \ldots , \beta_{k-1},[\alpha_1, \ldots, \alpha_{n-1}, \beta_k], \ldots, \beta_n]
\end{align*}
whenever the $1$-forms $\alpha_i \in \Omega^1(M)$ are closed, $1 \leqslant i \leqslant n-1$ and for any $\beta_j$.
\item $[\Pi^{\sharp}(\alpha_1 \wedge \cdots \wedge \alpha_{n-1}), \Pi^{\sharp}(\beta_1 \wedge \cdots \wedge \beta_{n-1})]\\ 
= \sum_{k=1}^{n-1} \Pi^{\sharp} (\beta_1 \wedge\cdots \wedge [\alpha_1,\ldots , \alpha_{n-1}, \beta_k] \wedge \cdots \wedge \beta_{n-1})$\\
for closed $1$-forms $\alpha_i \in \Omega^1 (M)$ and for any $1$-forms $\beta_j$. 
\end{enumerate}

The Nambu-form bracket on  $\Omega^1(M)$, together with the usual Lie algebroid structure on $TM$ yields an example of a notion called a {\it weak Lie-Filippov algebroid pair of order $n$}, $n>2,$ on a smooth vector bundle (cf. Definition $5.5$, \cite{bas-bas-das-muk}). 

In order to classify such structures, the authors formulate a notion of {\it Nambu-Gerstenhaber algebra of order $n$}. It turns out, weak-Lie-Filippov algebroid pair structures of order $n$, $n>2$, on a smooth vector bundle $A$ over $M$, are in bijective correspondence with Nambu-Gerstenhaber brackets of order $n$ on the graded commutative, associative algebra $\Gamma \bigwedge^\bullet A^*$ of multisections of $A^*$, where $A^*$ is the dual bundle (cf. Definition $5.7$, Theorem $5.8$, \cite{bas-bas-das-muk}).

Moreover, for a Nambu-Poisson manifold $M$ of order $n > 2$, the Nambu-Gerstenhaber bracket on $\Omega^\bullet(M)$, extending the Nambu-form bracket on $\Omega^1(M)$ satisfies certain suitable compatibility condition similar to the compatibility condition of a Lie bialgebroid. This motivates the authors to introduce the notion of a {\it weak Lie-Filippov bialgebroid structure} of order $n$ on a smooth vector bundle.

\begin{defn}\label{lie-fill-defn}
A {\it weak Lie-Filippov bialgebroid of order $n>2$} over a smooth manifold $M$ consists of a pair $(A, A^*)$, where $A$ is a smooth vector bundle over $M$ with dual bundle $A^*$ satisfying the following properties:
\begin{enumerate}
\item   $A$ is a Lie algebroid with $d_A$ being the differential of the Lie algebroid cohomology of $A$ with trivial representation;
\item the space of smooth sections $\Gamma A^*$ admits a skew-symmetric $n$-ary bracket
$$[~, \ldots ,~]: {\Gamma A^* \times \cdots \times \Gamma A^*} \longrightarrow \Gamma A^*$$
satisfying 
$$[\alpha_1, \ldots , \alpha_{n-1}, [\beta_1, \ldots , \beta_n]] = \sum_{k=1}^n [\beta_1, \ldots , \beta_{k-1}, [\alpha_1, \ldots , \alpha_{n-1}, \beta_k], \ldots ,\beta_n]$$
for all $d_A$-closed sections $\alpha_i \in \Gamma A^*,~ 1 \leqslant i \leqslant n-1$ and for any sections $\beta_j \in \Gamma A^*,~ 1\leqslant j\leqslant n;$
\item there exists a vector bundle map $\rho : \bigwedge^{n-1}A^* \longrightarrow TM$, called the {\it anchor} of the pair $(A, A^*)$, such that the identity
$$ [\rho (\alpha_1 \wedge \cdots \wedge \alpha_{n-1}), \rho (\beta_1 \wedge \cdots \wedge \beta_{n-1})] = \sum_{k=1}^{n-1}\rho (\beta_1 \wedge \cdots \wedge [\alpha_1, \ldots , \alpha_{n-1}, \beta_k] \wedge \cdots \wedge \beta_{n-1})$$
holds for all $d_A$-closed sections $\alpha_i \in \Gamma A^*,~ 1\leqslant i \leqslant n-1$ and for any sections $\beta_j \in \Gamma A^*,~ 1\leqslant j\leqslant n-1;$
\item for all sections $\alpha_i \in \Gamma A^*,~ 1\leqslant i \leqslant n$ and any $f \in C^\infty(M)$,
$$[\alpha_1, \ldots , \alpha_{n-1}, f\alpha_n] = f [\alpha_1, \ldots , \alpha_{n-1}, \alpha_n] + \rho (\alpha_1 \wedge \cdots \wedge \alpha_{n-1})(f)\alpha_n$$ holds;
\item the following compatibility condition holds:
$$ d_A[\alpha_1, \ldots , \alpha_n] = \sum_{k=1}^n [\alpha_1, \ldots , d_A\alpha_k, \ldots , \alpha_n],$$ for any $\alpha_i \in \Gamma A^\ast,$ $1 \leqslant i \leqslant n$,  where the bracket $[~, \ldots , ~]$ on the right hand side is the graded extension of the bracket on  $\Gamma A^\ast$.
\end{enumerate}
\end{defn}

A weak Lie-Filippov bialgebroid (of order $n$) over $M$ is denoted by $(A, A^*)$ when all the structures are understood. A Lie bialgebroid is a Lie-Filippov
bialgebroid of order $2$ such that the conditions (2) and (3) of the above definition has no restriction on $\alpha$.

In \cite{bas-bas-das-muk}, the authors have shown that for a Nambu-Poisson manifold $M$ of order $n > 2$, the pair $(TM, T^*M)$ is a weak Lie-Filippov bialgebroid of order $n$ (cf. Corollary $6.3$, \cite{bas-bas-das-muk}). It is also proved that if $(G, \Pi)$ is a Nambu-Lie group \cite{vais} of order $n$ with its Lie algebra $\mathfrak{g},$ then $(\mathfrak{g}, \mathfrak{g}^*)$ forms a (weak) Lie-Filippov bialgebroid of order $n$ over a Point.

It is known that the base of a Lie bialgebroid carries a natural Poisson structure. In \cite{bas-bas-das-muk} it has been extended to the Nambu-Poisson set up.

\begin{prop} \label{wlfb-base-nambu}(\cite{bas-bas-das-muk})
Let $(A, A^*)$ be a weak Lie-Filippov bialgebroid (of order n) over $M$. Then the bracket
\begin{align*}
 \{f_1, \ldots, f_n\}_{(A, A^*)} := \rho (d_A f_1 \wedge \cdots \wedge d_A f_{n-1}) f_n
\end{align*}
defines a Nambu-Poisson structure of order n on $M$.
\end{prop}

It is known that, given a coisotropic submanifold $C$ of a Poisson manifold $M$, the conormal bundle $(TC)^0 \rightarrow C$ is a Lie subalgebroid
of the cotangent Lie algebroid $T^*M$ \cite{wein}. If $M$ is a Nambu-Poisson manifold of order $n$ ($n \geqslant 3$), the cotangent bundle $T^*M$ is not a Filippov
algebroid. However we have the following useful result.

\begin{prop}\label{coiso-n-bracket}
 Let $C$ be a closed embedded coisotropic submanifold of a Nambu-Poisson manifold $(M, \Pi)$ of order $n$. Then
 \begin{enumerate}
\item the bundle map $\Pi^{\sharp} : \bigwedge^{n-1}T^*M \rightarrow TM$ maps $\bigwedge^{n-1} (TC)^0$ to $TC$;
\item the Nambu-form bracket on the space of $1$-forms $\Omega^1(M)$ can be restricted to the sections of the conormal bundle $ (TC)^0 \rightarrow C$.
\end{enumerate}
\end{prop}
\begin{proof}
The assertion $(1)$ follows from the definition of coisotropic submanifold. To prove $(2)$, let $\alpha_1, \ldots ,\alpha_n \in \Gamma (TC)^0.$ We extend them to $1$-forms on $M$, which we  denote by the same notation. Let $X \in \mathcal{X}^1(M)$ be such that $X\big|_C$ is tangent to $C$. From the definition
of Nambu-form bracket on $1$-forms, we have
\begin{align*}
\langle [\alpha_1, \ldots, \alpha_n], X \rangle = \sum_{k=1}^n (-1)^{n-k} \langle \mathcal{L}_{\Pi^{\sharp}(\alpha_1 \wedge \cdots \wedge \widehat{\alpha}_k \wedge \cdots \wedge \alpha_n)} \alpha_k, X \rangle
- (n-1) \langle d (\Pi(\alpha_1,\ldots, \alpha_n)), X \rangle. 
\end{align*}
Observe that
\begin{align*}
 \langle \mathcal{L}_{\Pi^{\sharp}(\alpha_1 \wedge \cdots \wedge \widehat{\alpha}_k \wedge \cdots \wedge  \alpha_n)}  \alpha_k, X \rangle 
=& \mathcal{L}_{ \Pi^\sharp (\alpha_1 \wedge \cdots \wedge \widehat{\alpha}_k \wedge \cdots \wedge \alpha_n )} \langle \alpha_k, X \rangle \\
&- \langle \alpha_k, [\Pi^{\sharp}(\alpha_1 \wedge \cdots \wedge \widehat{\alpha}_k \wedge \cdots \wedge \alpha_n), X]\rangle .
\end{align*}
This is zero on $C$, because, 
\begin{itemize}
\item $\langle \alpha_k, X\rangle$ is zero on $C$; 
\item $\Pi^{\sharp}(\alpha_1 \wedge \cdots \wedge \widehat{\alpha_k} \wedge \cdots \wedge \alpha_n)$ and $X$ are both tangent to $C$ and hence their Lie bracket is also tangent to C. Thus its pairing with $\alpha_k$ vanish on $C$.
\end{itemize}
Note that $\Pi^{\sharp}(\alpha_1 \wedge \cdots \wedge \alpha_{n-1})\big|_C$ is tangent to $C$ and $\alpha_n \big|_C \in (TC)^0.$ As a consequence, the function
\begin{align*}
\Pi (\alpha_1,\ldots, \alpha_n) = \langle \alpha_n, \Pi^{\sharp}(\alpha_1 \wedge \cdots \wedge \alpha_{n-1})\rangle
\end{align*}
is zero on $C.$  Therefore, the differential
$d (\Pi(\alpha_1, \ldots, \alpha_n))$ restricted to $C$ is in $(TC)^0$, which in turn implies that the second term of the right hand side also vanish on $C$. Hence
\begin{align*}
 [\alpha_1, \ldots, \alpha_n]\big|_C \in (TC)^0.
\end{align*}
One can check that the restriction to $C$ does not depend on the chosen extension. Hence it defines a bracket on the sections of the conormal bundle $(TC)^0 \rightarrow C$.
\end{proof}

\begin{remark}
\begin{enumerate}
\item Let $m_0 \in M$ such that $\Pi (m_0) = 0$. Then $\{m_0\}$ is a coisotropic submanifold of $M$. In this case, the conormal structure becomes $T_{m_0}^*M$, which is a Filippov algebra.
\item The Nambu structure of a Nambu-Lie group G vanishes at the identity element and therefore the dual $\mathfrak{g}^*$ of the Lie algebra $\mathfrak{g}$ of G has a Filippov algebra structure \cite{vais}.
\end{enumerate}
\end{remark}

\begin{remark}\label{nlg-bracket-anchor}
Let $(G \rightrightarrows M, \Pi)$ be a Nambu-Lie groupoid of order $n$ with Lie algebroid $AG \rightarrow M$. By the Proposition \ref{coiso-n-bracket}, we see that the space of sections of the conormal bundle $ A^*G = (TM)^0 \rightarrow M$ admits a skew-symmetric $n$-bracket $[~,\ldots,~]$ and there exists a bundle map
$$\rho := \Pi^{\sharp}\big|_{{\bigwedge}^{n-1} (TM)^0} : {\bigwedge}^{n-1} A^*G = {\bigwedge}^{n-1} (TM)^0 \rightarrow TM ,$$ as $M$ is a coisotropic submanifold of $G$.
\end{remark}

Let $(G \rightrightarrows M, \Pi)$ be a Nambu-Lie groupoid of order $n$ with Lie algebroid $AG \rightarrow M$.
 Let $f \in C^\infty(M).$ Then by part ($5$) of the Theorem \ref{multiplicative}, $\iota_{d(\beta^*f)} \Pi$ is a left invariant $(n-1)$-vector field
on $G$. Therefore, there exists an $(n-1)$-multisection $\delta^0_\Pi (f)\in \Gamma{\bigwedge^{n-1}AG}$ of the Lie algebroid $AG$ such that
\begin{align*}
 \iota_{d(\beta^*f)} \Pi = \overleftarrow{\delta^0_\Pi (f)} .
\end{align*}
Then we have the following result.

\begin{prop}
 Let $(G \rightrightarrows M, \Pi)$ be a Nambu-Lie groupoid of order $n$ and $AG \rightarrow M$ be its Lie algebroid.  Then for any $ X \in \Gamma AG$,
\begin{align*}
 \mathcal{L}_{\overleftarrow{X}} \Pi := [\overleftarrow{X}, \Pi]
\end{align*}
is a left invariant $n$ vector field on $G$, where $\overleftarrow{X}$ is the left invariant vector field on $G$ corresponding to $X$. Moreover $\mathcal{L}_{\overleftarrow{X}} \Pi$ corresponds to the $n$-multisection $- \delta^1_\Pi (X) \in \Gamma{\bigwedge^n AG}$,
that is,
\begin{align*}
 \mathcal{L}_{\overleftarrow{X}} \Pi = -\overleftarrow{\delta^1_\Pi (X)}
\end{align*}
where $\delta^1_\Pi (X) \in \Gamma{\bigwedge^n AG}$ is given by
\begin{align*}
\delta^1_\Pi (X) (\alpha_1, \ldots, \alpha_n) = \sum_{k=1}^{n} (-1)^{n-k} \Pi^{\sharp} (\alpha_1 \wedge \cdots \wedge \hat{\alpha}_k \wedge \cdots \wedge \alpha_n) (X(\alpha_k))
- X([\alpha_1,\ldots, \alpha_n])
\end{align*}
for $\alpha_1, \ldots, \alpha_n \in \Gamma A^*G = \Gamma (TM)^0.$
\end{prop}

\begin{proof}
Let $\mathcal{X}_t = \text{exp} \hspace*{0.05cm} tX$ be the one-parameter family of bisections generated by $X \in \Gamma AG.$ Let $g \in G$ with $\beta(g) = u$. Let
$u_t = (\text{exp} \hspace*{0.05cm} tX) (u)$ be the integral curve of $\overleftarrow{X}$ starting from $u$. If $\mathcal{G}$ is any (local) bisection through $g$, then from the multiplicativity condition of $\Pi$ (cf. Theorem \ref{multiplicative}), we have
\begin{align*}
 \Pi (  g u_t) =   (r_{\mathcal{X}_t})_* \Pi(g) + (l_{\mathcal{G}})_* \Pi(u_t)-  (r_{\mathcal{X}_t})_* (l_{\mathcal{G}})_* \Pi(u).
\end{align*}
Therefore,
\begin{align*}
(r_{\mathcal{X}^{-1}_t})_* \Pi ( g u_t) - \Pi(g) =  (r_{\mathcal{X}^{-1}_t})_*   (l_{ \mathcal{G}})_* \Pi(u_t) - (l_{\mathcal{G}})_* \Pi (u).
\end{align*}
Taking derivative at $t = 0$, one obtains
\begin{align*}
 (\mathcal{L}_{\overleftarrow{X}} \Pi)(g) = (l_{\mathcal{G}})_* ((\mathcal{L}_{\overleftarrow{X}} \Pi)(u)).
\end{align*}
Therefore, $\mathcal{L}_{\overleftarrow{X}} \Pi $ is left invariant by the Proposition \ref{left-invariant} and hence it corresponds to some $n$-multisection of $AG$. To show that
$\mathcal{L}_{\overleftarrow{X}} \Pi$ corresponds to $- \delta^1_\Pi (X) \in \Gamma {\bigwedge^n AG}$, we have to check that $\mathcal{L}_{\overleftarrow{X}} \Pi$ and $- \overleftarrow{\delta^1_\Pi (X)}$ coincide on the unit space $M$ (both being left invariant). Since both of them are tangent to $\alpha$-fibers, it is enough to show that they coincide on the conormal bundle $(TM)^0$.
Let $\alpha_1, \ldots, \alpha_n$ be any sections of $(TM)^0$ and $\tilde{\alpha}_1, \ldots, \tilde{\alpha}_n$ be their respective extensions to one forms on $G$. Observe that
\begin{align*}
& (\mathcal{L}_{\overleftarrow{X}} \Pi)\big|_M (\alpha_1, \ldots, \alpha_n)\\
 =& \big[\langle \overleftarrow{X}, d (\Pi(\tilde{\alpha}_1,\ldots,\tilde{\alpha}_n)) \rangle - \sum_{k=1}^{n} \Pi(\tilde{\alpha}_1 ,\ldots, \mathcal{L}_{\overleftarrow{X}} \tilde{\alpha}_k,\ldots, \tilde{\alpha}_n)\big]\big|_M\\
 =& \big[\langle \overleftarrow{X}, [\tilde{\alpha}_1,\ldots,\tilde{\alpha}_n] \rangle - \sum_{k=1}^n (-1)^{n-k} \langle \overleftarrow{X}, \iota_{\Pi^{\sharp} (\tilde{\alpha}_1 \wedge\cdots \wedge \hat{\tilde{\alpha}}_k \wedge\cdots \wedge \tilde{\alpha}_n)} d\tilde{\alpha}_k \rangle\\
  & - \sum_{k=1}^{n} (-1)^{n-k} \langle \Pi^{\sharp} (\tilde{\alpha}_1 \wedge\cdots \wedge \hat{\tilde{\alpha}}_k \wedge\cdots \wedge \tilde{\alpha}_n), \mathcal{L}_{\overleftarrow{X}} \tilde{\alpha}_k \rangle \big]\big|_M\\
  & \hspace*{5cm} (\text{from the Equation}~ (\ref{nambu-bracket}))\\
 =& \big[\langle \overleftarrow{X}, [\tilde{\alpha}_1,\ldots,\tilde{\alpha}_n] \rangle - \sum_{k=1}^n (-1)^{n-k} \langle \Pi^{\sharp} (\tilde{\alpha}_1 \wedge\cdots \wedge \hat{\tilde{\alpha}}_k \wedge\cdots \wedge \tilde{\alpha}_n), d \iota_{\overleftarrow{X}} \tilde {\alpha}_k \rangle \big]\big|_M\\
  & \hspace*{5cm} (\text{using Cartan formula})\\
 =& \langle X, [\alpha_1,\ldots,\alpha_n] \rangle - \sum_{k=1}^n (-1)^{n-k} \Pi^{\sharp} ({\alpha}_1 \wedge\cdots \wedge \hat{\alpha}_k \wedge\cdots \wedge {\alpha_n}) (X (\alpha_k))\\
 = & -\delta^1_\Pi (X) (\alpha_1,\ldots, \alpha_n).
\end{align*}
\end{proof}

To make our notation simple, let us denote $\delta^0_\Pi, \delta^1_\Pi$ by the same symbol $\delta_\Pi$. We extend $\delta_\Pi$ to the graded algebra $\Gamma{\bigwedge^{\bullet}A}$ of multisections of $AG$ by the following rule
\begin{align*}
 \delta_\Pi (P \wedge Q) = \delta_\Pi (P) \wedge Q + (-1)^{|P| (n-1)} P \wedge \delta_\Pi (Q)
\end{align*}
for $P \in \Gamma{\bigwedge^{|P|}A}, Q \in \Gamma{\bigwedge^{|Q|}A}$. Then the operator
\begin{align*}
 \delta_\Pi : \Gamma{{\bigwedge}^kAG} \rightarrow \Gamma{{\bigwedge}^{k+n -1} AG}
\end{align*}
satisfies
\begin{align*}
 \delta_\Pi ([P,Q]) = [ \delta_\Pi (P), Q] + (-1)^{(|P|-1)(n-1)} [P, \delta_\Pi (Q)].
\end{align*}
Note that the operator $\delta_\Pi$ need not satisfy condition $\delta_\Pi \circ \delta_\Pi = 0 .$

\medskip

We known that, Lie bialgebroids are infinitesimal form of Poisson groupoids. More precisely, given a Poisson groupoid $G \rightrightarrows M$ with Lie algebroid $AG$, it is known that its dual bundle $A^*G$
also carries a Lie algebroid structure and $(AG, A^*G)$ forms a Lie bialgebroid.  In the next theorem we show that weak Lie-Filippov bialgebroids are infinitesimal form of Nambu-Lie groupoids.




\begin{thm}\label{nambu-grpd-bialgbd}
Let $(G \rightrightarrows M , \Pi) $ be a Nambu-Lie groupoid of order $n$ with Lie algebroid $AG \rightarrow M$. Then $(AG, A^*G)$ forms a weak Lie-Filippov bialgebroid of order $n$ over $M$.
\end{thm}

\begin{proof}
From the Remark \ref{nlg-bracket-anchor}, we have the space of sections of the bundle $A^*G = (TM)^0 \rightarrow M$ admits a skew-symmetric $n$-bracket
$[~, \ldots, ~]$ and there exists a bundle map $$\rho : {\bigwedge}^{n-1}A^*G \rightarrow TM.$$

Let $\alpha_1, \ldots, \alpha_{n-1} \in \Gamma (TM)^0=\Gamma{(A^*G)} $ with $d_A \alpha_i = 0$, for all $i= 1, \ldots ,n-1.$ Let $\tilde{\alpha}_1, \ldots, \tilde{\alpha}_{n-1}$ be, respectively, their extensions to left invariant $1$-forms on $G$ such that $d \tilde{\alpha}_i = 0$, for all
$i= 1, \ldots ,n-1.$ Then the conditions $(2)$ and $(3)$ of the Definition \ref{lie-fill-defn} of a weak Lie-Filippov algebroid pair follows from the weak Lie-Filippov bialgebroid structure $(TG, T^*G)$ (Note that $G$ is a Nambu-Poisson manifold). 

Let $f \in C^\infty(M).$ Then observe that
\begin{align*}
 [\tilde {\alpha}_1, \ldots, \tilde{\alpha}_{n-1}, (\beta^*f) \tilde{\alpha}_n ] = (\beta^*f) [\tilde{\alpha}_1, \ldots, \tilde{\alpha}_{n-1},  \tilde{\alpha}_n ] + \Pi^{\sharp}(\tilde{\alpha}_1 \wedge \cdots \wedge \tilde{\alpha}_{n-1})(\beta^*f)\tilde{\alpha}_n.
\end{align*}
Since $(\beta^*f) \tilde{\alpha}_n = \widetilde{f \alpha}_n$, by assertion $(2)$ of the Propostion \ref{coiso-n-bracket}, we get
\begin{align*}
[\alpha_1,\ldots, \alpha_{n-1}, f \alpha_n] = f [\alpha_1, \ldots, \alpha_n] + \rho (\alpha_1 \wedge \cdots \wedge \alpha_{n-1}) (f) \alpha_n
\end{align*}
proving condition (4) of the Definition \ref{lie-fill-defn}. Moreover the compatibility condition of the weak Lie-Filippov bialgebroid (condition $(5)$ of the Definition \ref{lie-fill-defn}) follows from the observation that for any $\alpha \in \Gamma{(A^*G)} = \Gamma (TM)^0$ and any left invariant extension $\tilde{\alpha} \in \Omega^1(G)$, we have
\begin{align*}
 d_A \alpha = (d \tilde \alpha)|_M.
\end{align*}
Thus, $(AG, A^*G)$ is a weak Lie-Filippov bialgebroid of order $n$.
\end{proof}

\begin{remark}If $(G, \Pi)$ is a Nambu-Lie group with Lie algebra $\mathfrak{g}$, the dual vector space $\mathfrak{g}^*$
carries a Filippov algebra structure \cite{vais}. Moreover the pair $(\mathfrak{g} , \mathfrak{g}^*)$ forms a (weak) Lie-Filippov bialgebra (\cite{bas-bas-das-muk, vais}).
The Lie-Filippov bialgebra $(\mathfrak{g}, \mathfrak{g}^*)$ is the infinitesimal form of the Nambu-Lie group $(G, \Pi).$ 
A Lie-Filippov bialgebra $(\mathfrak{g}, \mathfrak{g}^*)$ can also be seen a Lie 
algebra $\mathfrak{g}$ together with a Filippov algebra structure on the dual vector space $\mathfrak{g}^*$ such that the map $\delta : \mathfrak{g} \rightarrow \bigwedge^n \mathfrak{g}$ dual
to the Filippov bracket on $\mathfrak{g}^*$, defines a $1$-cocycle of $\mathfrak{g}$ with respect to the adjoint representation on $\bigwedge^n \mathfrak{g}$.
\end{remark}

We have seen that given a Nambu-Lie groupoid of order $n$, there is an induced Nambu-Poisson structure on the base manifold (cf. Proposition \ref{inverse-basenambu}). On the other hand, 
given a weak Lie-Filippov bialgebroid, there is an induced Nambu-Poisson structure on the base (cf. Theorem \ref{wlfb-base-nambu}).
The next proposition compares these Nambu-Poisson structures on the base induced from the Nambu Lie groupoid and its infinitesimal.

\begin{prop}\label{compare-two-NP-structure}
Let $(G \rightrightarrows M , \Pi) $ be a Nambu-Lie groupoid (of order $n$) with associated weak Lie-Filippov bialgebroid $(AG, A^*G).$ Then the induced Nambu structures on $M$
coming from the Nambu-Lie groupoid and the weak Lie-Filippov bialgebroid are related by
\begin{align*}
 \{ ~, \ldots, ~\}_M  = (-1)^{n-1}\{ ~, \ldots, ~\}_{(AG, A^*G)} .
\end{align*}
\end{prop}

\begin{proof}
For any functions $f_1, \ldots, f_n \in C^\infty(M)$, we have
\begin{align*}
 \{f_1, \ldots, f_n\}_{(AG, A^*G)} =& \Pi^{\sharp}\big|_M (d_A f_1 \wedge \cdots \wedge d_A f_{n-1}) f_n\\
=& \Pi^{\sharp} ( d (\beta^* f_1) \wedge \cdots \wedge d (\beta^* f_{n-1}))\big|_M f_n\\
=& \Pi (\beta^* f_1 , \ldots, \beta^* f_{n-1}, \beta^* f_n)\big|_M\\
=& (-1)^{n-1} \big(  \beta^*\{f_1, \ldots, f_n\} \big) \big|_M\\
=& (-1)^{n-1} \{f_1, \ldots, f_n\}_M.
\end{align*}
\end{proof}
\begin{remark}
 It is known that under some connectedness and simply connectedness assumption, any Lie bialgebra integrates to a Poisson-Lie group \cite{lu-wein}, and any Lie bialgebroid
integrates to a Poisson groupoid \cite{mac-xu2}. These results does not hold in the context of Nambu structures of order $\geqslant 3$. Let $G$ be a connected and simply-connected Lie group with Lie algebra $\mathfrak{g}.$ Given a Lie-Filippov bialgebra structure $(\mathfrak{g}, \mathfrak{g}^*)$ on $\mathfrak{g}$, the $1$-cocycle $\delta : \mathfrak{g} \rightarrow \bigwedge^n \mathfrak{g}$
dual to the Filippov algebra bracket on $\mathfrak{g}^*$ integrates a multiplicative $n$-vector field $\Pi$ on the Lie group. However this $n$-vector field (for $n \geqslant 3$) need not be a Nambu tensor \cite{vais}, that is, need not be locally decomposable. Thus (weak) Lie-Filippov bialgebra
does not integrate to a Nambu-Lie group in general. 
\end{remark}

\section{Coisotropic subgroupoids of a Nambu-Lie groupoid}

In this final section, we introduce the notion of coisotropic subgroupoid of a Nambu-Lie groupoid 
 and study the infinitesimal object corresponding to it.

\begin{defn}
 Let $(G \rightrightarrows M , \Pi)$ be a Nambu-Lie groupoid of order $n$. Then a subgroupoid $H \rightrightarrows N$ is called a {\it coisotropic subgroupoid} if $H$ is a coisotropic submanifold of $G$ with respect to $\Pi$.
\end{defn}

\begin{exam}
\begin{enumerate}
\item For $n=2$, that is, when $G \rightrightarrows M$ is a Poisson groupoid, this notion is same as the coisotropic subgroupoid of a Poisson groupoid introduced in \cite{xu}.
\item Let $(G, \Pi)$ be a Nambu-Lie group. Then a subgroup of $G$ is called coisotropic if it is also a coisotropic submanifold of $G$. Any coisotropic subgroup
of $G$ is a coisotropic subgroupoid over a point.
\item Let $(G \rightrightarrows M , \Pi)$ be a Nambu-Lie groupoid. Then by the Proposition \ref{inverse-basenambu}, there exist an induced Nambu-structure
on $M$ for which the source map $\alpha$ is a Nambu-Poisson map. Let $N \hookrightarrow M$ be a coisotropic submanifold of $M$ with respect to this induced Nambu
structure. Consider the restriction $G\big|_N := \alpha^{-1}(N) \cap \beta^{-1}(N)$, then $G\big|_N \rightrightarrows N$ is a coisotropic subgroupoid.
\item Let $G \rightrightarrows M$ be a Nambu-Lie groupoid. If the set of all elements of $G$ which has same source and target, is a submanifold of $G$, then it is a coisotropic subgroupoid.
\end{enumerate}
\end{exam}

Note that, the infinitesimal object corresponding to a Nambu-Lie groupoid $(G\rightrightarrows M , \Pi)$ is the weak Lie-Filippov bialgebroid
$(AG, A^*G)$. Therefore it is natural to ask how the Lie algebroid of a coisotropic subgroupoid $H \rightrightarrows N$ is related to the weak Lie-Filippov bialgebroid $(AG, A^*G)$. To answer this question, we  introduce a notion of {\it coisotropic subalgebroid} of a weak Lie-Filippov bialgebroid and show that infinitesimal forms of  coisotropic subgroupoids of a Nambu-Lie groupoid $(G \rightrightarrows M, \Pi)$ appear as coisotropic subalgebroids of the corresponding weak Lie-Filippov bialgebroid $(AG, A^*G)$.

\begin{defn}\label{coiso-subalg}
 Let $(A, A^*)$ be a weak Lie-Filippov bialgebroid of order $n$ over $M$. Then a Lie subalgebroid $B \rightarrow N$ of $A\rightarrow M$ is
called a {\it coisotropic subalgebroid} if the anchor $\rho: \bigwedge^{n-1}A^* \rightarrow TM$ and the $n$-bracket $[~, \ldots, ~]$ on $\Gamma{A^*}$
satisfy the following properties.
\begin{enumerate}
\item The anchor $\rho$ maps $\bigwedge^{n-1}B^{0} \rightarrow TN$.
\item If $\alpha_1, \ldots, \alpha_n \in \Gamma{A^*}$ with ${\alpha_i}\big|_N \in B^{0}$ for all $i$,  then $[\alpha_1, \ldots, \alpha_n]\big|_N \in B^{0}$.
\item If $\alpha_1, \ldots, \alpha_n \in \Gamma{A^*}$ with ${\alpha_i}\big|_N \in B^{0}$ for all $i$ and $\alpha_n\big|_N = 0$, then $[\alpha_1, \ldots, \alpha_n]\big|_N = 0,$
where $B_x^{0} = \{\gamma \in A_x^* | \gamma (v) = 0, \forall v \in B_x\}$, is the annihilator of $B_x$, $x \in N$.
\end{enumerate}
\end{defn}

\begin{exam}
 Let $M$ be a Nambu-Poisson manifold, then $(TM, T^*M)$ is a weak Lie-Filippov bialgebroid over $M$. Let $N \hookrightarrow M$ be a coisotropic submanifold. Then from the Proposition \ref{coiso-n-bracket}, it follows that the tangent bundle $TN \rightarrow N$ is a coisotropic subalgebroid. 
\end{exam}

It is known that (Proposition \ref{wlfb-base-nambu}, see also \cite{bas-bas-das-muk}), the base of a weak Lie-Filippov bialgebroid carries a Nambu structure.
The next Proposition shows that the base of a coisotropic subalgebroid is a coisotropic submanifold with respect to this induced Nambu structure.

\begin{prop}
 Let $(A, A^*)$ be a weak Lie-Filippov bialgebroid over $M$ and $B \rightarrow N$ be a coisotropic subalgebroid. Then $N$ is a coisotropic submanifold of $M$.
\end{prop}

\begin{proof}
Let $a : A \rightarrow TM$ denote the anchor of the Lie algebroid $A$ and $\rho : \bigwedge^{n-1}A^* \rightarrow TM$ be the anchor of pair $(A, A^*).$
We first show that, $a^* (TN)^0 \subseteq B^{0}$. This is true because, $ \langle a^*\xi_x, v \rangle = \langle \xi_x, a(v)\rangle = 0$
for $\xi_x \in (TN)_x^0$ and $v \in B_x$.

Let $\Pi_{(A, A^*)}$ be the induced Nambu structure on $M$ coming from the weak Lie-Filippov bialgebroid $(A, A^*).$ Then the induced map
$\Pi^{\sharp}_{(A, A^*)} : \bigwedge^{n-1}T^*M \rightarrow TM$ is given by $\Pi^{\sharp}_{(A, A^*)} = \rho \circ \bigwedge^{n-1}a^*$. Therefore, for any
$\xi_1, \ldots, \xi_{n-1} \in (TN)^0$, we have
$$ \Pi^{\sharp}_{(A, A^*)} (\xi_1, \ldots, \xi_{n-1}) = \rho (a^*\xi_1, \ldots, a^*\xi_{n-1}) \in TN $$
as $a^*\xi_i \in B^{0}$ and $B$ is a coisotropic subalgebroid. Therefore $N$ is a coisotropic submanifold of $M$.
\end{proof}

The next proposition shows that the infinitesimal object corresponding to coisotropic subgroupoids are coisotropic subalgebroids.

\begin{prop}\label{coiso-subgrpd-coiso-subalgbd}
Let $(G \rightrightarrows M, \Pi)$ be a Nambu-Lie groupoid with weak Lie-Filippov bialgebroid $(AG, A^*G)$. Let $H \rightrightarrows N$
be a coisotropic subgroupoid of $G \rightrightarrows M$ with Lie algebroid $AH \rightarrow N$. Then $AH \rightarrow N$ is a coisotropic subalgebroid.
\end{prop}

\begin{proof}
 Since $H \rightrightarrows N$ is a Lie subgroupoid of $G \rightrightarrows M$, therefore $AH \rightarrow N$ is a Lie subalgebroid of
$AG \rightarrow M$. We claim that the anchor $\rho = \Pi^{\sharp}\big|_{\bigwedge^{n-1}(TM)^0} = \Pi^{\sharp}\big|_{\bigwedge^{n-1} (A^*G)}$ of the weak
Lie-Filippov bialgebroid $(AG, A^*G)$ maps $\bigwedge^{n-1}(AH)^0$ to $TN$. First observe that, for any $x \in N$, $(AH)_x^0 = (TM)_x^0 \cap (TH)_x^0$
and $T_xN = T_xM \cap T_xH$. Therefore, 
$\rho$ maps $\bigwedge^{n-1} (AH)^0$ to 
$$\Pi^{\sharp}({\bigwedge}^{n-1} (TM)^0) \cap \Pi^{\sharp}({\bigwedge}^{n-1} (TH)^0) \subseteq TM \cap TH \cong TN,$$ 
here we have used the fact that $M$ and $H$ are both coisotropic submanifolds of $G$.

Let $\alpha_1, \ldots, \alpha_n \in \Gamma{A^*G} = \Gamma{(TM)^0}$ such that $\alpha_i\big|_N \in (AH)^0$, for all $i=1, \ldots, n$. Let
$\tilde{\alpha}_1, \ldots, \tilde{\alpha}_n$ be one forms on $G$ extending $\alpha_1, \ldots, \alpha_n$ and are conormal to $H$. Then
by the Proposition \ref{coiso-n-bracket}, the $1$-form $[\tilde{\alpha_1}, \ldots, \tilde{\alpha_n}]$ is conormal to both $M$ and $H$, as $M$
and $H$ are both coisotropic submanifolds of $G$. Therefore, 
$$[\tilde{\alpha_1}, \ldots, \tilde{\alpha_n}]\big|_N \in (TM)^0 \cap (TH)^0 \cong (AH)^0.$$ 
Verification of the last condition of the Definition \ref{coiso-subalg} is similar. Hence $AH \rightarrow N$ is a coisotropic subalgebroid of $(AG, A^*G).$
\end{proof}

\begin{corollary}
 Let $(G \rightrightarrows M , \Pi)$ be a Nambu-Lie groupoid and $H \rightrightarrows N$ be a coisotropic subgroupoid. Then $N$ is a coisotropic submanifold of $M$.
\end{corollary}


\providecommand{\bysame}{\leavevmode\hbox to3em{\hrulefill}\thinspace}
\providecommand{\MR}{\relax\ifhmode\unskip\space\fi MR }
\providecommand{\MRhref}[2]{%
  \href{http://www.ams.org/mathscinet-getitem?mr=#1}{#2}
}
\providecommand{\href}[2]{#2}

\end{document}